\documentclass{amsart}
\usepackage{eucal,amssymb,latexsym,esint,xcolor}
\theoremstyle{plain}
\newtheorem{theorem}{Theorem}

\newtheorem{lemma}{Lemma}

\newtheorem{proposition}{Proposition}
\usepackage{pgf}
\usepackage{color}
\theoremstyle{remark}
\theoremstyle{definition}
\numberwithin{equation}{section}
\newcommand{\e}{\epsilon}

\newcommand{\R}{\mathbb R}

\newcommand{\C}{\mathbb C}

\newcommand{\Rn}{\mathbb R^n}

\begin{document}
\title[Kato square root problem with unbounded leading coefficients]{Kato Square Root Problem with unbounded leading coefficients}

\author[L. Escauriaza]{Luis Escauriaza}
\address[L. Escauriaza]{UPV/EHU, Dpto. Matem\'aticas, Barrio Sarriena s/n 48940 Leioa, Spain}
\email{luis.escauriaza@ehu.eus}
\thanks{L. Escauriaza is supported by grants MTM2014-53145-P and IT641-13 (GIC12/96).}

\author[S. Hofmann]{Steve
Hofmann}
\address[S. Hofmann]{Department of Mathematics,
University of
Missouri-Columbia,
Columbia, MO 65211}
\email{hofmann@math.missouri.edu}
\thanks{S. Hofmann is supported by NSF grant no. DMS-1664047. 
This material is based upon work supported by the National Science Foundation 
under Grant No. DMS- 1440140 while the authors were in residence at the Mathematical 
Sciences Research Institute in Berkeley, California, during the Spring 2017 semester.}


\begin{abstract} We prove the Kato conjecture for elliptic operators, $L=-\nabla\cdot\left((\mathbf A+\mathbf D)\nabla\ \right)$, with $\mathbf A$ a complex measurable bounded coercive matrix and $\mathbf D$ a measurable real-valued skew-symmetric matrix in $\R^n$ with entries in $BMO(\R^n)$;\,  i.e., the domain of $\sqrt{L}\,$ is the Sobolev space $\dot H^1(\R^n)$ in any dimension, with the estimate $\|\sqrt{L}\, f\|_2
\lesssim \| \nabla f\|_2$.
\end{abstract}
%\today
\maketitle

\begin{section}{Introduction}\label{S:1}
Let $\mathbf A=(a_{ij})$ be an $n\times n$ matrix of complex, $L^\infty$ coefficients defined on $\mathbb{R}^n$ and satisfying the ellipticity or accretivity condition
\begin{equation}\label{eq1.1} 
\lambda |\xi |^2 \leq \Re \langle\mathbf A \xi, \xi
\rangle  \equiv \Re 
\sum_{i,j}a_{ij}(x) \, \xi_j \, \overline{\xi}_i \, , \,\,\,
\|\mathbf A\|_{\infty} \leq \lambda^{-1} ,\end{equation} for $\xi \in \mathbb{C}^n$ and
for
some $0<\lambda \le 1$.  We consider a 
divergence form operator
\begin{equation}\label{eq1.2} Lu\equiv -\nabla\cdot\left(\mathbf A(x)\nabla u\right).\end{equation}
The accretivity condition \eqref{eq1.1} enables one to
define a square root $\sqrt{L}$ 
\cite{K2}
and a fundamental issue was \lq\lq solve the square root problem\rq\rq, i.e. to establish the estimate
\begin{equation}\label{eq1.3}
\| \sqrt{ L}\, f\|_{L^2(\mathbb{R}^n)}\leq N\| \nabla f\|_{L^2(\mathbb{R}^n)},
\end{equation} 
with $N$ depending on $n$ and $\lambda$. The latter estimate is connected with the question of the analyticity of the mapping $\mathbf A\to \sqrt{L}$, which in turn has applications to the perturbation theory for certain classes of hyperbolic equations  \cite{Mc1}.  We remark that
\eqref{eq1.3} is equivalent to the opposite 
inequality for the
square root of the adjoint
operator $L^{*}.$

In \cite{K1,K2} Kato conjectured that an abstract version of \eqref{eq1.3} might hold, for \lq\lq regularly accretive operators\rq\rq.  A 
counter-example
to this abstract conjecture was obtained by McIntosh 
\cite{Mc2},
who then reformulated the
conjecture in the following form, 
bearing in mind that
Kato's interest in the problem had been motivated by the 
special case
of elliptic differential operators:

{\it The estimate \eqref{eq1.3} holds for $L$ defined as in \eqref{eq1.2}, for any $L^\infty$, $n\times n$ matrix $\mathbf A$ with 
complex entries, for which \eqref{eq1.1} holds.}

To establish the validity of this conjecture became known as the Kato Problem or square root problem. In 1982 it was solved in one dimension \cite{CMcM}, where it is essentially equivalent to the problem of proving the $L^2$ boundedness of the Cauchy integral operator on Lipschitz curves \cite{KM}.  

For $ n>1,$ a restricted version of the
conjecture, also essentially posed by Kato
in \cite{K2}, was proved by P. Auscher, 
S. Hofmann, J.L. Lewis and P. Tchamitchian in \cite{Au01}. The restricted version treated the case that $\mathbf A$ is close in the $L^{\infty}$ norm to a real symmetric matrix of bounded measurable coefficients. It is this version that 
yields the perturbation
results for hyperbolic equations alluded to above \cite{Mc1}. 

Prior to the latter result, 
the conjecture was proved in
higher dimensions when $\|\mathbf A-\mathbf I\|_{L^\infty(\R^n)}\le \epsilon(n)$ \cite{CDM, FJK1,FJK2,J}. \cite{DJ} gave a different proof using the $T(1)$
theorem. Sharper bounds for the constant $\epsilon (n)$ on the order of $n^{-\frac{1}{2}}$ were obtained in \cite{J}. In \cite{AT} it was proved  when $\|\mathbf A-\mathbf I\|_{BMO(\R^n)}$ is small. 

Later, the validity of the
conjecture was established when the heat kernel of the operator $L$ 
satisfies the \lq\lq Gaussian\rq\rq property, first in 2 dimensions \cite{HMc}, and then in all dimensions \cite{HLMc}; i.e. let 
$G(x,y,t)$ denote the kernel of the operator $e^{-tL}$, we say that
$L$ 
satisfies the Gaussian property, if there are positive constants
$0<\alpha\le 1$ and $N$ such that
\begin{align*} 
(i) &\quad |G(x,y,t)|\leq N t^{-\frac n2} e^{ 
-|x-y|^2/Nt}\\
(ii) & \quad
|G(x+h,y,t)-G(x,y,t)|+|Gx,y+h,t)-G(x,y,t)|\\
&\quad \le N \left( |h|/\sqrt{t}\right)^{\alpha}t^{-\frac n2}e^{
-|x-y|^2/Nt},
\end{align*}
where
the latter holds when $t>0$ and either, $|h| \leq t,$ or $|h| \leq |x-y|/2 $.

The Gaussian property holds when $\mathbf A$ is real-valued by 
results
of Aronson
\cite{Ar} and in some cases for complex $\mathbf A$:
in two dimensions from \cite{AMcT2} and for perturbations of 
real operators %or even for small coefficients in $BMO$-norm 
\cite{Au}.   
 Hence, \cite{HMc}, \cite{HLMc} solve the conjecture in the 
 former two dimensional cases or the latter $n$-dimensional cases. 
 
 Finally, the conjecture was solved for general complex, 
 bounded and coercive matrices $\mathbf A$ satisfying \eqref{eq1.1} in \cite{AHLMT}. 

The purpose of this note is to show that 
minor modifications of the reasoning in \cite{AHLMT} also yield the following extension.  Let 
\begin{equation*}
H^1(\mathbb{R}^n):= \{f\in L^2(\mathbb{R}^n):\, \nabla f \in L^2(\mathbb{R}^n)\}
\end{equation*} denote the usual
Sobolev space, and $\dot{H}^1(\mathbb{R}^n)$ its homogeneous version; i.e.,
$\dot{H}^1$ is the closure of $C_0^\infty(\mathbb{R}^n)$ with respect to the seminorm
$\|f\|_{\dot{H}^1(\mathbb{R}^n)}:= \|\nabla f\|_{L^2(\mathbb{R}^n)}$.

\begin{theorem}\label{thsecondorder} For any operator
\begin{equation}\label{E: forma especial}
L=-\nabla\cdot\left((\mathbf A(x)+\mathbf D(x))\nabla\ \right)
\end{equation}
with $\mathbf A$ a bounded complex-valued coercive matrix satisfying \eqref{eq1.1} and $\mathbf D$ a real-valued skew-symmetric matrix in $\R^n$ with entries in $BMO(\R^n)$ satisfying \eqref{E: condicion2}, the domain of $\sqrt{L}\,$ contains $H^1(\R^n)$ and \eqref{eq1.3} holds over $\dot H^1(\R^n)$.
\end{theorem}

We remark in passing that the Gaussian property has been shown to hold with $\alpha=\alpha(\lambda,n)$ and 
$N=N(\lambda, n)$, when 
\begin{equation*}
Lu\equiv -\nabla\cdot\left(\mathbf (\mathbf A(x)+\mathbf D(x))\nabla u\right),
\end{equation*}
with $\mathbf A$ a real-valued, bounded symmetric and coercive matrix satisfying \eqref{eq1.1} and $\mathbf D=(d_{ij}(x))$ a real-valued skew-symmetric $BMO(\R^n)$ matrix \cite{ZhongminGuangyu17, SereginSilvestreSverakZlatos12} with
\begin{equation}\label{E: condicion2}
\|\mathbf D\|_{BMO(\R^n)}\le \lambda^{-1}.
\end{equation}
The arguments of \cite{HLMc} could  be modified to treat this restricted case.
On the other hand, as in \cite{AHLMT}, we do not require the Gaussian property in the
proof of Theorem \ref{thsecondorder} in the present paper.

We recall that a function $\beta:\R^n\longrightarrow\C$ is in $BMO(\R^n)$ or has bounded mean oscillation \cite{JohnNirenberg61}, when it is is locally integrable and  
\begin{equation*}
\|\beta\|_{BMO(\R^n)}=\sup_{Q\subset\R^n}\fint_Q|\beta-m_Q(\beta)|\,dx<+\infty,
\end{equation*}
where $Q$ ranges over all cubes in $\R^n$ with sides parallel to the coordinate axis and
\begin{equation*}
m_Q(\beta)=\fint_Q\beta\,dx.
\end{equation*}
If one defines other norms by
\begin{equation*}
\|\beta\|_{BMO(\R^n)_p}=\sup_{Q\subset\R^n}\left(\fint_Q|\beta-m_Q(\beta)|^p\,dx\right)^{\frac 1p},\ 1<p<\infty,
\end{equation*}
the John-Nirenberg inequality \cite{JohnNirenberg61} implies that all the $BMO(\R^n)_p$-norms are equivalent, when $1\le p<\infty$. Finally, $BMO(\R^n)$ is the dual of $H^1_{at}(\R^n)$, the real Hardy space in $\R^n$, where $f$ in $L^1(\R^n)$ is in $H^1_{at}(\R^n)$, when
\begin{equation*}
\sup_{\epsilon>0}|\theta_{\epsilon}\ast f|
\end{equation*}
is in $L^1(\R^n)$, where $\theta_\epsilon=\epsilon^{-n}\theta(x/\epsilon)$ and $\theta$ is any smooth non-negative compactly supported mollifier with integral equal to $1$ \cite{FeffermanCStein}. In particular, when $\beta$ is in $BMO(\R^n)$ and $f$ is in $H^1_{at}(\R^n)$, the principal value of the integral of $\beta f$ is well defined \cite{Latter} and
\begin{equation*}
\left| \int_{\R^n}\beta f\, dx\right| \le N(n)\|\beta\|_{BMO(\R^n)}\|f\|_{H^1_{at}(\R^n)},
\end{equation*}
with
\begin{equation*}
\|f\|_{H^1_{at}(\R^n)}=\|\sup_{\epsilon>0}|\theta_{\epsilon}\ast f|\|_{L^1(\R^n)}.
\end{equation*}

Following \cite{SereginSilvestreSverakZlatos12}, when
\begin{equation}\label{E: dvergencia0}
L=-\nabla\cdot\left(\mathbf A\nabla\ \right)-\mathbf b\cdot\nabla
\end{equation}
with $\mathbf A$ a complex-valued bounded matrix verifying \eqref{eq1.1} and $\mathbf b$ a real-valued divergence-free vector field with
\begin{equation}\label{E: condicicion}
\sup_{x\in\R^n,\, r>0}r\fint_{B_r(x)}|\mathbf b|\, dx< \infty,
\end{equation} 
the matrix 
\begin{equation*}
\mathbf D = \Delta^{-1}( \nabla\mathbf b-\nabla\mathbf b^{\top})
\end{equation*}
is skew-symmetric and real-valued matrix with
\begin{equation*}
\|\mathbf D\|_{BMO(\R^n)}\le N(n)\sup_{x\in\R^n,\, r>0}r\fint_{B_r(x)}|\mathbf b|\, dx
\end{equation*}
and $L$ can be written in the form \eqref{E: forma especial}. Thus, according with Theorem \ref{thsecondorder}, the domain of the square root of the accretive operator \eqref{E: dvergencia0} contains $H^1(\R^n)$, when $\mathbf A$ is as above, $\mathbf b$ is real-valued and \eqref{E: condicicion} holds.

The proof of Theorem \ref{thsecondorder} requires simple modifications to the original reasonings in \cite{AHLMT} but most importantly the following two compensated compactness-type results.
\begin{proposition}\label{E: compensated compactness} When $f,g:\R^n\longrightarrow\R$ are in $H^1(\R^n)$ and $i,j\in\{1,\dots,n\}$, the functions $\partial_if\partial_jg-\partial_jf\partial_i g$ and $f\partial_i f$ are in $H^1_{at}(\R^n)$ and there is $N=N(n)$ such that
\begin{equation}\label{E:3}
\|\partial_if\partial_jg-\partial_jf\partial_i g\|_{H^1_{at}(\R^n)}\le N\|\nabla f\|_{L^2(\R^n)}\|\nabla g\|_{L^2(\R^n)}
\end{equation}
and
\begin{equation}\label{E:4}
\|f\partial_i f\|_{H^1_{at}(\R^n)}\le N\|f\|_{L^2(\R^n)}\|\nabla f\|_{L^2(\R^n)}.
\end{equation} 
\end{proposition}
The reader can find the proofs of \eqref{E:3} and \eqref{E:4} in \cite{CoifmanLionsMeyerSemmes93} and \cite{ZhongminGuangyu17} respectively. 

In the next section we explain the minor modifications one must make to the reasonings in the proof of the Conjecture 1.4  in \cite{AHLMT}, to derive its extension in Theorem \ref{thsecondorder}. Throughout the next pages $N$ denotes a constant which depends at most on $\lambda$ and $n$, $B_r$ an open ball in $\R^n$ of radius $r>0$, $Q$ a cube in $\R^n$ with sides parallel to the coordinate axis, $x_Q$ its center and $\delta(Q)$ its side length.

\end{section}

\begin{section}{Proof of Theorem \ref{thsecondorder}}

Setting
\begin{equation*}
\langle f\text{,}\ g\rangle=\int_{\R^n} f(x)\overline g(x)\,dx,\ \text{for}\ f,g\in L^2(\R^n),
\end{equation*}
it follows from H\"older's inequality, \eqref{eq1.1}, \eqref{E: condicion2}, the identity
\begin{equation}\label{E_skesymmetry}
\int_{\R^n}\mathbf D(x)\nabla u\cdot\nabla\overline v\,dx=\frac 12\int_{\R^n}d_{ij}(x)\left(\partial_iu\partial_j\overline v-\partial_ju\partial_i\overline v\right)\,dx,
\end{equation}
\eqref{E:3} and the fact that 
\begin{equation*}
\Re{\mathbf D(x)\xi\cdot\overline\xi}=0,\  \text{when}\ x\in R^n\ \text{and}\ \xi\in\C^n,
\end{equation*}
that the sesquilinear form
\begin{equation*}
\mathcal B(u,v)=\int_{\R^n}\left(\mathbf A(x)+\mathbf D(x)\right)\nabla u\cdot\nabla\overline v\, dx,
\end{equation*}
associated to the unbounded operator  $L$ in \eqref{E: forma especial}, with domain 
\begin{equation*}
\mathcal D(L)=\{ u\in H^1(\R^n): Lu\in L^2(\R^n)\,\}
\end{equation*}
and by the relation
\begin{equation*}
\langle Lu,v\rangle= {\mathcal B(u,v)},\ \text{when}\ u\in\mathcal D(L)\ \text{and}\ v\in H^1(\R^n),
\end{equation*}
 is bounded and coercive on $H^1(\R^n)$ with
\begin{equation*}
|\mathcal B(u,v)|\le N\|\nabla u\|_{L^2(\R^n)}\|\nabla v\|_{L^2(\R^n)}
\end{equation*}
and
\begin{equation*}
\Re{\mathcal B(u,u)}\ge \lambda\int_{\R^n}|\nabla u|^2\,dx,\ \text{when}\ u\in H^1(\R^n),
\end{equation*}
when the matrices $\mathbf A$ and $\mathbf D$ satisfy  the conditions in Theorem \ref{thsecondorder}. $L$ is an accretive unbounded operator,
\begin{equation*}
\Re\langle Lu,u\rangle =\Re\mathcal B(u,u)\ge 0,\ \text{when}\ u\in\mathcal D(L);
\end{equation*}
 $L$ is also m-accretive \cite[p. 279]{K1} and the operators
 \begin{equation}\label{E: operadores basicos}
(1+t^2L)^{-1},\quad t\nabla(1+t^2L)^{-1},\quad (1+t^2L)^{-1}t\,\nabla\cdot\quad \text{and}\quad
t^2\nabla(1+t^2L)^{-1}\nabla\cdot
\end{equation}
are uniformly  $L^2(\R^n)$-bounded with
bounds depending $n$ and $\lambda$ for all $t>0$, where
\begin{equation*}
u=\left(1+t^2L\right)^{-1}f\quad \text{and}\quad  w=\left(1+t^2L\right)^{-1}t\nabla\cdot \mathbf f
\end{equation*}
are the unique Lax-Milgram weak solutions in $H^1(\R^n)$ satisfying respectively
\begin{equation}\label{E:formabilineal}
\int_{\R^n} u\overline v+t^2\left(\mathbf A+\mathbf D\right)\nabla u\cdot\overline{\nabla v}\, dx =\int_{\R^n}f\,\overline v\, dx
\end{equation}
and
\begin{equation}\label{E:formabilineal3}
\int_{\R^n} w\overline v+t^2\left(\mathbf A+\mathbf D\right)\nabla w\cdot\overline{\nabla v}\, dx =-t\int_{\R^n}\mathbf f\cdot\nabla \overline v\, dx,
\end{equation}
for all $v$ in $C_0^\infty(\R^n)$, when $f$ and $\mathbf f$ are in $L^2(\R^n)$. Similar bounds hold when $L$ is replaced above by the adjoint of $L$
\begin{equation*}
L^{*}= -\nabla\cdot\left(\left(\mathbf A^{*}-\mathbf D\right)\nabla\  \right),
\end{equation*}
where $\mathbf A^{*}$ denotes the transpose conjugate matrix of $\mathbf A$.

Following T. Kato \cite[p. 281]{K1}, $L$ has a unique m-accretive square root $\sqrt{L}$ given by
\begin{equation*}
\sqrt{L}f=\tfrac 1\pi\int_0^{+\infty}\lambda^{-\frac 12}(\lambda+L)^{-1}Lf\,d\lambda,\ \quad\text{when}\ f\in\mathcal D(L).
\end{equation*}
The identities
 \begin{equation*}
 (\lambda+L)^{-k-1}=-\frac 1k\frac{d}{d\lambda}(\lambda+L)^{-k}, \quad \lambda >0,
 \end{equation*}
 for $k= 1, 2$, integration by parts and the change of variables $\lambda=1/t^2$, show that
 \begin{equation*}
\sqrt{L}f=\tfrac 8\pi\int_0^{+\infty}(1+t^2L)^{-3}t^3L^2f\, \tfrac{dt}{t},\ \quad\text{when}\ f\in\mathcal D(L^2)=(1+L)^{-1}\mathcal D(L)
\end{equation*}
and as in \cite{AHLMT}, we use the latter resolution formula for $\sqrt{L}$ to prove Theorem \ref{thsecondorder}.  

As is \cite{AHLMT}, Theorem \ref{thsecondorder} follows once \eqref{eq1.3} is derived for $f$ in a dense subspace of $H^1(\R^n)$, as $\mathcal D(L^2)$ (here, $L^2 = L\,\circ\, L$); because then \eqref{E: paso fundamental}, \eqref{E: otra importante} and the closedness of $\sqrt{L}$ as an unbounded operator over $L^2(\R^n)$, show that $H^1(\R^n)$ is contained  the domain of $\sqrt{L}$ and \eqref{eq1.3} holds for $f$ in $H^1(\R^n)$. Finally, $H^1(\R^n)$ is dense in $\dot H^1(\R^n)$ and $\sqrt{L}$ can be uniquely extended by density to $\dot H^1(\R^n)$.

We have
\begin{equation*}
\begin{split}
|\langle \sqrt{L}f,g\rangle |&=\left |\int_0^{+\infty}\langle\left(1+t^2L\right)^{-1}tLf\text{,}\left(1+t^2L^{*}\right)^{-2}t^2L^{*}g\rangle\tfrac{dt}t\right |\\
&\le\left(\int_0^{+\infty}\|\left(1+t^2L\right)^{-1}tLf\|_{L^2(\R^n)}^2\tfrac{dt}t\right)^{\frac12}\\
&\times \left(\int_0^{+\infty}\|\left(1+t^2L^{*}\right)^{-2}t^2L^{*}g\|_{L^2(\R^n)}^2\tfrac{dt}t\right)^{\frac 12}
 \end{split}
\end{equation*}
and
\begin{equation}\label{E: unanecesaria}
\left(\int_0^{+\infty}\|\left(1+t^2L^{*}\right)^{-2}t^2L^{*}g\|_{L^2(\R^n)}^2\tfrac{dt}t\right)^{\frac12}\le N\|g\|_{L^2(\R^n)}.
\end{equation}
To verify the later inequality, define 
\begin{equation*}
\mathcal S_t=\left(1+t^2L\right)^{-2}t^2L= \left(1+t^2L\right)^{-1}-\left(1+t^2L\right)^{-2}.
\end{equation*}
By duality
\begin{equation*}
\begin{split}
&\int_0^{+\infty}\|\left(1+t^2L^{*}\right)^{-2}t^2L^{*}g\|_{L^2(\R^n)}^2\tfrac{dt}t=\langle \int_0^{+\infty}\mathcal S_t\mathcal S_t^{*}g\tfrac{dt}t,g\rangle\\
&\le \|\int_0^{+\infty}\mathcal S_t\mathcal S_t^{*}g\tfrac{dt}t\|_{L^2(\R^n)}\|g\|_{L^2(\R^n)}
\end{split}
\end{equation*}
and because the operators $\mathcal S_t$ are uniformly bounded in $\mathcal B(L^2(\R^n))$ and
\begin{equation}\label{E:otras acotaciones}
\|\mathcal S_t^{*}\mathcal S_s\|_{\mathcal B(L^2(\R^n))}\le N\min{\{t/s,s/t\}},\ \text{when}\  s,t >0,
\end{equation}
Cotlar's Lemma for integrals \cite{Cotlar74,Stein93} and \eqref{E:otras acotaciones} imply that 
\begin{equation*}
\|\int_0^{+\infty}\mathcal S_t\mathcal S_t^{*}g\tfrac{dt}t\|_{L^2(\R^n)}\le N\|g\|_{L^2(\R^n)},
\end{equation*}
which gives \eqref{E: unanecesaria}. Thus,
\begin{equation}\label{E: paso fundamental}
\|\sqrt{L}f\|_{L^2(\R^n)}\le N\left(\int_0^{+\infty}\|\left(1+t^2L\right)^{-1}tLf\|_{L^2(\R^n)}^2\tfrac{dt}t\right)^{\frac12},
\end{equation}
provided that \eqref{E:otras acotaciones} holds. Write then for $t,s >0$,
\begin{equation*}
\begin{split}
&\langle\mathcal S^{*}_t\mathcal S_s f,g\rangle\\
&= \langle\left(1+s^2L\right)^{-2} s^2Lf,\left(1+t^2L\right)^{-2}t^2Lg\rangle\\
&= \langle s^2L\left(1+s^2L\right)^{-2} f,\left(1+t^2L\right)^{-1}g-\left(1+t^2L\right)^{-2}g\rangle\\
&= -\frac st\ \langle \left(\mathbf A+\mathbf D\right)s\nabla\left(1+s^2L\right)^{-2}f, t\nabla\left(1+t^2L\right)^{-1}g-t\nabla\cdot\left(1+t^2L\right)^{-2}g \rangle\\
&= -\frac ts\ \langle s\nabla\left(1+s^2L\right)^{-1}f-s\nabla\cdot\left(1+s^2L\right)^{-2}f, \left(\mathbf A+\mathbf D\right)t\nabla\left(1+t^2L\right)^{-2}g \rangle
\end{split}
\end{equation*}
and use H\"older's inequality, \eqref{E_skesymmetry}, \eqref{E: operadores basicos} and \eqref{E:3} to derive  \eqref{E:otras acotaciones}, from the previous identities.

The next goal is to show that the operator 
\begin{equation}\label{L: eloperador}
\mathcal T_t=\left(1+t^2L\right)^{-1}t^2L =I-\left(1+t^2L\right)^{-1}
 \end{equation}
 has Gaffney bounds and a well defined action over $L^\infty(\R^n)$ and the space of Lipschitz functions over $\R^n$. To show it we prove first the following Lemma.
 \begin{lemma}\label{L: Gaffney}
 There are $\theta=\theta(\lambda,n)$ and $N$ such that the inequalities
\begin{equation*}
\|e^{x\cdot\xi/t}\left(1+t^2L\right)^{-1}f\|_{L^2(\R^n)}+\|e^{x\cdot\xi/t} t\nabla\left(1+t^2L\right)^{-1}f\|_{L^2(\R^n)}\le N\|e^{x\cdot\xi/t}f\|_{L^2(\R^n)},
\end{equation*}
\begin{equation*}
\begin{split}
\|e^{x\cdot\xi/t}\left(1+t^2L\right)^{-1}t\nabla\cdot\mathbf f\|_{L^2(\R^n)}+\|e^{x\cdot\xi/t} t^2\nabla\left(1+t^2L\right)^{-1}\nabla\cdot\mathbf f&\|_{L^2(\R^n)}\\
&\le N\|e^{x\cdot\xi/t}\mathbf f\|_{L^2(\R^n)},
\end{split}
\end{equation*}
hold when  $\xi$ is in $\R^n$ and $|\xi|\le\theta$.
 \end{lemma}
 \begin{proof}
 We first prove the Lemma when the domain of $L$ is replaced by
 \begin{equation*}
\mathcal D(L)=\{ f\in H^1_0(\Omega): Lf\in L^2(\Omega)\},
\end{equation*}
where $\Omega$ is a bounded domain in $\R^n$ and $L^2(\R^n)$ is replaced by $L^2(\Omega)$ in Lemma \ref{L: Gaffney}. In that case, when $f$ is in $L^2(\Omega)$, $u=\left(1+t^2L\right)^{-1}f$   is the unique Lax-Milgram weak solution in $H^1_0(\Omega)$, which satisfies
\begin{equation}\label{E:formabilineal2}
\int_{\Omega} u\overline v+t^2\left(\mathbf A+\mathbf D\right)\nabla u\cdot\overline{\nabla v}\, dx =\int_{\Omega}f\,\overline v\, dx,
\end{equation}
for all $v$ in $H^1_0(\Omega)$. Take then $v=e^{2x\cdot\xi/t}u$ in \eqref{E:formabilineal2} to find that
\begin{equation*}
\int_{\Omega}e^{2x\cdot\xi/t} \left[|u|^2+t^2\left(\mathbf A+\mathbf D\right)\nabla u\cdot\left(\nabla\overline u+2(\xi/ t)\overline u\right)\right]\, dx=\int_{\Omega}e^{2x\cdot\xi/t}f\overline u\,dx.
\end{equation*} 
Taking real parts, we get
\begin{equation*}
\begin{split}
&\|e^{x\cdot\xi/t}u\|_{L^2(\Omega)}^2+\lambda\|e^{x\cdot\xi/t}t\nabla u\|_{L^2(\Omega)}^2\le \|e^{x\cdot\xi/t}f\|_{L^2(\Omega)}\|e^{x\cdot\xi/t}u\|_{L^2(\Omega)}\\
&+\lambda^{-1}|\xi|\|e^{x\cdot\xi/t}t\nabla u\|_{L^2(\Omega)}\|e^{x\cdot\xi/t}u\|_{L^2(\Omega)}-t\int_\Omega\mathbf D\nabla\left(|e^{x\cdot\xi/t}u|^2\right)\cdot\xi\,dx.
\end{split}
\end{equation*}
According with \eqref{E:4} and because $e^{x\cdot\xi/t}u$ is in $H^1_0(\Omega)\subset H^1(\R^n)$, the absolute value of the last integral above is bounded by
\begin{equation*}
\lambda^{-1}|\xi|\|e^{x\cdot\xi/t}u\|_{L^2(\Omega)}\|e^{x\cdot\xi/t}t\nabla u\|_{L^2(\Omega)}+|\xi|^2\|e^{x\cdot\xi/t}u\|_{L^2(\Omega)}^2
\end{equation*}
and the inequality
 \begin{equation}\label{E: casicai}
 \|e^{x\cdot\xi/t}u\|_{L^2(\Omega)}+\|e^{x\cdot\xi/t} t\nabla u\|_{L^2(\Omega)}\le N\|e^{x\cdot\xi/t}f\|_{L^2(\Omega)},
 \end{equation}
 follows, when $|\xi|\le\theta$, $\xi$ is in $R^n$ and $\theta$ is sufficiently small.
  
 For $f$ in $C_0^\infty(\R^n)$ with the support of $f$ contained in $B_R$, let $u_R$ denote the Lax-Milgram weak solution to \eqref{E:formabilineal2}, when $\Omega=B_R$. Because $\mathbf D\nabla\varphi$ is in $L^2(\R^n)$, when $\varphi$ is in $C_0^\infty(\R^n)$ and the bounds that we have for $u_R$ and $\nabla u_R$ are independent of $R\ge 1$, we can derive that $u_R$ converges to $u$ in $L^2_{loc}(\R^n)$ and $\nabla u_R$ converges weakly to $\nabla u$ in $L^2_{loc}(\R^n)$, where now $u= \left(1+t^2L\right)^{-1}f$ is as in Lemma \ref{L: Gaffney}. The first part of the Lemma now follows from \eqref{E: casicai} and the local weak convergences of $e^{x\cdot\xi/t}u_R$ and $e^{x\cdot\xi/t}\nabla u_R$ to $e^{x\cdot\xi/t}u$ and $e^{x\cdot\xi/t}\nabla u$, when $R$ tends to infinity.
 
The second part of the Lemma follows after replacing \eqref{E:formabilineal2} by \eqref{E:formabilineal3} and taking $v=e^{2x\cdot\xi/t}w$.
 \end{proof}
  \begin{lemma}\label{L: Gaffney2}
 There is $N$ such that the following inequalities hold for all cubes $Q$ in $\R^n$ with side length $\delta(Q)$ and $t>0$
 \begin{equation*}
 \|\left(1+t^2L\right)^{-1}f\|_{L^2(Q)}+\|t\nabla\left(1+t^2L\right)^{-1}f\|_{L^2(Q)}\le Ne^{-2^k\delta(Q)/Nt}\|f\|_{L^2(\R^n)},
 \end{equation*}
  \begin{equation*}
 \|\left(1+t^2L\right)^{-1}t\nabla\cdot\mathbf f\|_{L^2(Q)}+\|t^2\nabla\left(1+t^2L\right)^{-1}\nabla\cdot\mathbf f\|_{L^2(Q)}\le Ne^{-2^k\delta(Q)/Nt}\|\mathbf f\|_{L^2(\R^n)},
 \end{equation*}
 when  the supports of $f$ and $\mathbf f$ are contained in $2^{k+1}Q\setminus 2^kQ$ and $k\ge 1$.
 \end{lemma}
 \begin{proof}
 Without loss of generality we may assume that the cube $Q$ is centered at the origin and $2^kQ=\{x\in \R^n : \|x\|_\infty\le 2^{k-1}\delta(Q)\}$, when  $k\ge 1$. 
 Assume then that $f$ is supported in $2^{k+1}Q\setminus 2^kQ$ and write $\R^n\setminus\{0\}$ as the union of the sets, $\bigcup_{i=1}^nA_i\cup B_i$, where
\begin{equation*}
 A_i=\{x\in\R^n: \|x\|_\infty=x_i\}\ \text{and}\ B_i=\{x\in\R^n: \|x\|_\infty=-x_i\},\ i=1,\dots, n.
 \end{equation*}
 Then,
 \begin{equation*}
 \left(1+t^2L\right)^{-1}f=\sum_{i=1}^n\left(1+t^2L\right)^{-1}(f\chi_{A_i})+\left(1+t^2L\right)^{-1}(f\chi_{B_i})
 \end{equation*}
 and we show that the first inequality in the Lemma holds for each of these parts of $\left(1+t^2L\right)^{-1}f$. To get it for $\left(1+t^2L\right)^{-1}(f\chi_{A_1})$, apply Lemma \ref{L: Gaffney} to $f\chi_{A_1}$ with $\xi=-\theta e_1$, $e_1=\left(1,0,\dots,0\right)$ and observe that
 $e^{-\theta x_1/t}\ge e^{-\theta\delta(Q)/2t}$ inside $Q$ and $e^{-\theta x_1/t}\le e^{-\theta 2^k\delta(Q)/2t}$ inside $A_1\cap 2^{k+1}Q\setminus 2^kQ$.
 
 The other inequalities in Lemma \ref{L: Gaffney2} follow in the same way from Lemma \ref{L: Gaffney}.
 \end{proof}
For $f$ in $L^\infty(\R^n)$, define 
\begin{equation*}
\left(1+t^2L\right)^{-1}f=\lim_{R\to +\infty}\left(1+t^2L\right)^{-1}(f\chi_{B_R(x_0)}),
\end{equation*}
where $x_0$ is any point in $\R^n$ and the limit is taken in the $L^2_{loc}(\R^n)$-sense. The limit is well defined due to the Gaffney bounds in Lemma \ref{L: Gaffney2}, for if $x_1$ is any other point in $\R^n$, the symmetric difference between $B_R(x_0)$ and $B_R(x_1)$ is contained in an annulus $B_{2R}\setminus B_{\frac R2}$ for $R$ sufficiently large and
\begin{multline*}
\|\left(1+t^2L\right)^{-1}(f\chi_{B_R(x_0)})-\left(1+t^2L\right)^{-1}(f\chi_{B_R(x_1)})\|_{L^2(B_{\frac R4})}\\\le NR^{\frac n2}e^{-R/Nt}\|f\|_{L^\infty(\R^n)}.
\end{multline*}
Also, after writing
\begin{equation*}
\chi_{B_{R_2}\setminus B_{R_1}}=\chi_{B_{R_2}\setminus B_{2^{l+1}R_1}}+\sum_{i=0}^l
\chi_{B_{2^{i+1}R_1}\setminus B_{2^{i}R_1}}
\end{equation*}
when $2R<R_1<2^{l+1}R_1<R_2\le 2^{l+2}R_1$, it follows from Lemma \ref{L: Gaffney2} that
\begin{equation*}
\|\left(1+t^2L\right)^{-1}(f\chi_{B_{R_2}})-\left(1+t^2L\right)^{-1}(f\chi_{B_{R_1}})\|_{L^2(B_R)}\le Nt^{\frac n2 +1}R_1^{-1}\|f\|_{L^\infty(\R^n)},
\end{equation*}
which shows that $\left(1+t^2L\right)^{-1}(f\chi_{B_{R}})$ is a Cauchy sequence in $L^2_{loc}(\R^n)$, when $f$ is in $L^\infty(\R^n)$. Also, the Gaffney control that we have in Lemma \ref{L: Gaffney2} over the operator 
\begin{equation*}
t\nabla\left(1+t^2L\right)^{-1}
\end{equation*}
shows with similar reasonings that for $f$ in $L^\infty(\R^n)$,  $u=\left(1+t^2L\right)^{-1}f$ is a weak $H^1_{loc}(\R^n)$ solution over $\R^n$ to $u+t^2Lu=f$.

In particular, $\left(1+t^2L\right)^{-1}1=1$ and $\nabla\left(1+t^2L\right)^{-1}1=0$ in the above sense, because if $\eta_R(x)=\eta(x/R)$, with $\eta$ in $C_0^\infty(\R^n)$, $\eta=1$ in $B_1$ and $\eta=0$ outside $B_2$, $u_R= \left(1+t^2L\right)^{-1}(\eta_R)$ satisfies
\begin{equation*}
u_R-\eta_R+t^2L(u_R-\eta_R)=-t^2L\eta_R.
\end{equation*}
At the same time, the skew-symmetry of $\mathbf D$ implies that in the sense of distributions
\begin{equation*}
-t^2L\eta_R=t^2\nabla\cdot\left(\left(\mathbf A+\mathbf D-m_{B_{2R}}(\mathbf D)\right)\nabla\eta_R\right),\quad 
m_{B_{2R}}(\mathbf D)=\fint_{B_{2R}}\mathbf D\,dx.
\end{equation*}
Then,  the second inequality in Lemma \ref{L: Gaffney2} gives
\begin{equation*}
\begin{split}
&\|u_R-1\|_{L^2(B_{\frac R2})}+\|t\nabla u_R\|_{L^2(B_{\frac R2})}\\
&\le Nte^{-R/Nt}\|\mathbf A\nabla\eta_R+\left(\mathbf D-m_{B_{2R}}(\mathbf D)\right)\nabla\eta_R\|_{L^2(B_{2R})}\\&\le
NtR^{\frac n2 -1}e^{-R/Nt},
\end{split}
\end{equation*}
which tends to zero as $R$ tends to $+\infty$. The latter shows that the $L^2(\R^n)$-uniformly bounded operators $\mathcal T_t$ defined by \eqref{L: eloperador} verify Gaffney bounds, map $L^\infty(\R^n)$ into $L^2_{loc}(\R^n)$ and $\mathcal T_t(1)=0$, for $t>0$. 

For a Lipschitz function $f$ in $\R^n$, define in a similar manner
\begin{equation*}
\mathcal T_t(f)=\lim_{R\to +\infty}\mathcal T_t\left(\left(f-f(x_0)\right)\chi_{B_R(x_1)}\right),
\end{equation*}
where $x_0$ and $x_1$ are any points in $\R^n$. The limit is measured in the $L^2_{loc}(\R^n)$-sense and the definition is again independent of the choices of $x_0$ and $x_1$. Clearly, for $f$ Lipschitz, $\mathcal T_t(f)$ is a weak $H^1_{loc}(\R^n)$ solution over $\R^n$ to 
\begin{equation*}
\mathcal T_t(f)+t^2L\mathcal T_t(f)=t^2Lf.
\end{equation*}
This follows from the Gaffney bounds verified by the operators $\mathcal T_t$ and the following Lemma.

\begin{lemma}\label{L: otrolemita} Let $f$ be a Lipschitz function in $\R^n$ and $Q$ be a cube in $\R^n$ with $0<t\le\delta(Q)$. Then,
\begin{equation*}
\|\mathcal T_t(f)\|_{L^2(Q)}\le Nt|Q|^{\frac 12}\|\nabla f\|_{L^\infty(\R^n)}
\end{equation*}
and
\begin{equation*}
\|\nabla\mathcal T_t(f)\|_{L^2(Q)}\le N|Q|^{\frac 12}\|\nabla f\|_{L^\infty(\R^n)}.
\end{equation*}
\begin{proof} Let $x_Q$ denote the center of the cube $Q$. Write
\begin{equation*}
\begin{split}
&\mathcal T_t(f)=\lim_{R\to+\infty}\mathcal T_t\left(\left(f-f(x_Q)\right)\chi_{B_R(x_Q)}\right)\\
&=\mathcal T_t\big(\left(f-f(x_Q)\right)\eta_0\big)+
\sum_{k=0}^{+\infty}\mathcal T_t\big(\left(f-f(x_Q)\right)\left(\eta_{k+1}-\eta_k\right)\big),
\end{split}
\end{equation*}
where $\eta\in C_0^\infty(\R^n)$ is equal to $1$ in $2Q-x_Q$, $0$ outside $3Q-x_Q$ and $\eta_k(x)=\eta(x-x_Q/2^k)$, $k\ge 0$.
Then, the Gaffney bounds in Lemma \ref{L: Gaffney2} show that
\begin{multline}\label{E: que cosas}
\|\mathcal T_t(\left(f-f(x_Q)\right)\left(\eta_{k+1}-\eta_k\right))\|_{L^2(Q)}+\|t\nabla\mathcal T_t(\left(f-f(x_Q)\right)\left(\eta_{k+1}-\eta_k\right))\|_{L^2(Q)}\\
\le N2^{-k}\|\nabla f\|_{L^\infty(\R^n)}t|Q|^{1/2},
\end{multline}
when $k\ge 0$. Next, $u=\mathcal T_t\left(\left(f-f(x_Q)\right)\eta_0\right)$ is a weak $H^1(\R^n)$ solution to
\begin{equation*}
u+t^2Lu=-t^2\nabla\cdot\left(\mathbf A+\mathbf D\right)\nabla\left(\left(f-f(x_Q)\right)\eta_0\right)
\end{equation*}
and recalling that the distribution 
\begin{equation*}
\nabla\cdot\left(\mathbf D\nabla\left(\left(f-f(x_Q)\right)\eta_0\right)\right)
\end{equation*}
is the same as
\begin{equation*}
\nabla\cdot\left(\left(\mathbf D-m_Q(\mathbf D)\right)\nabla\left(\left(f-f(x_Q)\right)\eta_0\right)\right),
\end{equation*}
we find that 
\begin{equation*}
u=-\left(1+t^2L\right)^{-1}t^2\nabla\cdot\left[\left(\mathbf A+\left(\mathbf D - m_Q(\mathbf D)\right)\nabla\left(\left(f-f(x_Q)\right)\eta_0\right)\right)\right].
\end{equation*}
Then, the uniform boundedness of the last two operators in \eqref{E: operadores basicos} give
\begin{multline*}
\|u\|_{L^2(\R^n)}+\|t\nabla u\|_{L^2(\R^n)}\le Nt|Q|^{\frac 12}\left(1+\|\mathbf D\|_{BMO}\right)\|\nabla\left(\left(f-f(x_Q)\right)\eta_0\right)\|_{L^\infty(4Q)}\\
\le Nt|Q|^{\frac 12}\left(1+\|\mathbf D\|_{BMO}\right)\|\nabla f\|_{L^\infty(\R^n)},
\end{multline*}
and the Lemma follows after adding up \eqref{E: que cosas} and the last inequality.
\end{proof}
\end{lemma}
Next, we recall the following result in \cite[Lemma 3.9]{AAAHK}.
\begin{lemma}\label{L: elcachondeo} Let $\{\mathcal T_t: t>0\}$ be a family of bounded operators on $L^2(\R^n)$ satisfying for some $N>0$
\begin{enumerate}
\item $\sup_{t>0}{\|\mathcal T_t\|_{\mathcal B(L^2(\R^n))}}\le N$.
\item $\mathcal T_t$ verifies Gaffney bounds; i.e.  when $Q$ is a cube in $\R^n$ and $k\ge 1$
\begin{equation*}
 \|\mathcal T_t\left(f\chi_{2^{k+1}Q\setminus 2^kQ}\right)\|_{L^2(Q)}\le Ne^{-2^k\delta(Q)/Nt}\|f\chi_{2^{k+1}Q\setminus 2^kQ}\|_{L^2(\R^n)},
 \end{equation*}
 \item $\mathcal T_t(1)\equiv 0$ in $L^2_{loc}(\R^n)$.
 \end{enumerate}
 Then,
 \begin{equation*}
 \left(\int_{\R^{n+1}_+}|\tfrac 1t\mathcal T_t(f)|^2\tfrac{dxdt}t\right)^{\frac 12}\le N\left[1+\|\tfrac 1t\mathcal T_t(\Phi)\|_{C}\right]\|\nabla f\|_{L^2(\R^n)},
 \end{equation*}
 for all $f$ in $H^1(\R^n)$, where
 \begin{equation*}
 \|\tfrac 1t\mathcal T_t(\Phi)\|_{C}=\sup_{Q\subset\R^n}\left(\frac 1{|Q|}\int_{R_Q}|\tfrac 1t\mathcal T_t(\Phi)|^2\tfrac{dxdt}t\right)^{\frac 12},
 \end{equation*}
 $\Phi$ is the identity map of $\R^n$ and $R_Q$ the Carleson box $Q\times (0,\delta(Q))$.
\end{lemma}
Hence, as in \cite{AHLMT}, Lemma \ref{L: elcachondeo} implies that
\begin{equation}\label{E: otra importante}
\left(\int_0^{+\infty}\|\left(1+t^2L\right)^{-1}tLf\|_{L^2(\R^n)}^2\tfrac{dt}t\right)^{\frac12}\le N\|\nabla f\|_{L^2(\R^n)}
\end{equation}
after one shows with $\mathcal T_t$ as in \eqref{L: eloperador}, that the measure
\begin{equation*}
|\tfrac 1t\mathcal T_t(\Phi)|^2\tfrac{dxdt}t
\end{equation*}
is a Carleson measure with 
\begin{equation}\label{E: condición Carleson}
\|\tfrac 1t\mathcal T_t(\Phi)\|_{C}\le N
\end{equation} 
and \eqref{eq1.3} for $f$ in $\mathcal D(L^2)$ follows from \eqref{E: paso fundamental} and \eqref{E: otra importante}.

To obtain \eqref{E: condición Carleson}, it suffices to adapt the construction of \cite{HLMc} to verify a variant of the $T(b)$ theorem for square roots \cite{AT}: for a fixed cube $Q$ in $\R^n$, $0<\e< 1$ and $\xi$ a unit vector  in $\C^n$, define the scalar-valued function
\begin{equation}\label{E: elecion}
f^\epsilon_{Q,\xi}=\Phi_Q\cdot\xi-\mathcal T_t\left(\Phi_Q\cdot\xi\right),
\end{equation}
%\left(1+\left(\e \delta(Q)\right)^2L \right)^{-1}(\Phi_Q\cdot\xi),
with $\Phi_Q(x)=x-x_Q$ and $t=\e \delta(Q)$. Then, if follows from Lemma \ref{L: otrolemita} with $Q$ replaced by $10Q$, $t=\e\delta(Q)$ and $f=\Phi_Q\cdot\xi$ that
\begin{equation}\label{E: propiedad1}
\left(\fint_{10Q}|f^\epsilon_{Q,\xi}-\Phi_Q\cdot\xi|^2\,dx\right)^{\frac 12}\le N\e\delta(Q),
\end{equation}
\begin{equation}\label{E: propiedad2}
\left(\fint_{10Q}|\nabla f^\epsilon_{Q,\xi}-\xi|^2\,dx\right)^{\frac 12}\le N.
\end{equation}
Also $f^\epsilon_{Q,\xi}$ is a weak $H^1_{loc}(\R^n)$ solution to $f^\epsilon_{Q,\xi}+t^2Lf^\epsilon_{Q,\xi}=\Phi_Q\cdot\xi$ over $\R^n$, with $t=\e \delta(Q)$ and 
\begin{equation}\label{E: propiedad3}
\left(\fint_{10Q}|Lf^\epsilon_{Q,\xi}|^2\,dx\right)^{\frac 12}\le N/\left(\e\delta(Q)\right),
\end{equation}

The reasonings in \cite[Lemma 5.4]{AHLMT} show that given functions $f^\epsilon_{Q,\xi}$ in $H^1_{loc}(\R^n)$ verifying \eqref{E: propiedad1} and \eqref{E: propiedad2} for some $N>0$, there is $0<\e\le 1$, $\e=\e(N,n)$ and a finite set $W$ of unit vectors in $\C^n$, whose cardinality depends only on $\e$ and $n$, such that the inequality
\begin{equation}\label{E: ojofundamento}
\|\Psi\|_{C}\le N\sum_{\xi\in W}\sup_{Q\subset\R^n}\left(\frac1{|Q|}\int_{R_Q}|\Psi\cdot S^Q_t(\nabla f^\e_{Q,\xi})|^2\tfrac{dxdt}t\right)^{\frac 12},
\end{equation}
holds for all measurable functions $\Psi:\R^{n+1}\longrightarrow\C^n$ in $L^2_{loc}(\R^{n+1}_+)$, where for each $Q$ cube in $\R^n$, $S^Q_t$ denotes the dyadic averaging operator associated to the dyadic mesh generated by $Q$; i.e.
\begin{equation*}
S^Q_t(h)(x)=\fint_{Q'}h(y)\ dy,
\end{equation*}
for $x$ in the dyadic cube $Q'$ with $\frac 12\delta(Q')<t\le\delta(Q')$. In fact, the proof of \eqref{E: ojofundamento} in \cite[Lemma 5.4]{AHLMT} uses the compactness of the unit sphere in $\C^n$, properties of the distance function in $\C^n$, H\"older's inequality, the boundedness of the Hardy-Littlewood maximal function in $L^2(\Rn)$, a suitable stopping time argument independent of $\Psi$ and the interpolation inequality in \cite[Lemma 5.15]{AHLMT}. Thus, its proof  is independent of the choice of $\Psi$ and \eqref{E: ojofundamento} holds with $\Psi=\tfrac 1t\mathcal T_t(\Phi)$, when $L$ is as in Theorem \ref{thsecondorder} and for the choice of functions $f^\e_{Q,\xi}$ defined in \eqref{E: elecion}.

  Then, \eqref{E: condición Carleson} follows from \eqref{E: ojofundamento} with $\Psi=\tfrac 1t\mathcal T_t(\Phi)$ and Lemma \ref{E: fundamental2} below, which adapts \cite[Lemma 5.5]{AHLMT} to the more general hypothesis on the coefficients matrix of $L$ in Theorem \ref{thsecondorder}.
\begin{lemma}\label{E: fundamental2} Let $\e=\e(N,n)$ be the choice of $\e$ in \eqref{E: ojofundamento}. Then, there is $N>0$ such that
\begin{equation*}
\left(\frac1{|Q|}\int_{R_Q}|\tfrac1t\mathcal T_t(\Phi)\cdot S^Q_t(\nabla f^\e_{Q,\xi})|^2\tfrac{dxdt}t\right)^{\frac 12}\le N,
\end{equation*}
for all cubes $Q$ in $\R^n$ and $\xi$ a unit vector in $\C^n$
\end{lemma}
\begin{proof} Fix $Q$, $\xi$ in $\C^n$ with $|\xi|=1$ and make $\e=\e(N,n)$. 
Let $\chi$ be in $C_0^\infty(4Q)$ with $\chi=1$ in $2Q$, $\chi=0$ outside $3Q$ and
\begin{equation*}
\|\chi\|_\infty+\delta(Q)\|\nabla X\|_\infty\le N.
\end{equation*}
To simplify the notation set $f=f^{\e}_{Q,\xi}$ and $S_t=S^Q_t$. Then,
\begin{equation*}
\|\tfrac1t\mathcal T_t(\Phi)\cdot S_t(\nabla f)\|_{L^2(R_Q,dxdt/t)}= \|\tfrac1t\mathcal T_t(\Phi)\cdot S_t(\nabla\left(\chi f\right))\|_{L^2(R_Q,dxdt/t)} 
\end{equation*}
because $\nabla(\chi f)=\nabla f$ over $2Q$ and $S_t(\nabla f)$ only reads information about $\nabla f$ inside $Q$ to calculate its values at points $(x,t)$ in $R_Q$. Next, let $P_t$ denote the convolution with an even smooth mollifier, $\theta_t(x)=t^{-n}\theta(x/t)$, $\theta$ with integral $1$ and supported in $B_1$. We have
\begin{equation*}
\begin{split}
&\|\tfrac1t\mathcal T_t(\Phi)\cdot S_t(\nabla\left(\chi f\right))\|_{L^2(R_Q,dxdt/t)}\\
&\le \|\tfrac1t\mathcal T_t(\Phi)\cdot (S_t-P_t^2)\left(\nabla\left(\chi f\right)\right)\|_{L^2(\R^{n+1}_+,dxdt/t)}\\
&+\|\tfrac1t\left(\mathcal T_t(\Phi)\cdot \nabla P_t^2-\mathcal T_t\right) \left(\chi f\right)\|_{L^2(\R^{n+1}_+,dxdt/t)}+\|\tfrac1t\mathcal T_t(\chi f)\|_{L^2(R_Q,dxdt/t)}\\
&= I+II+III.
\end{split}
\end{equation*}
Then, $I$ and $II$ in the right hand side above are handled exactly as its analogues in \cite[Lemma 5.5]{AHLMT}. In particular, the only information about $f$ that one needs to bound $I$ and $II$ by $N\sqrt{|Q|}$ is that \eqref{E: propiedad1} and \eqref{E: propiedad2} imply the bound 
\begin{equation}\label{E: fundamental3}
\left(\fint_{5Q}|\nabla\left(\chi f\right)|^2\ dx\right)^{\frac 12}\le N
\end{equation} 
and it suffices to apply the same harmonic analysis techniques, which allow to handle the operators $\tfrac1t\mathcal T_t(\Phi)\cdot (S^Q_t-P_t^2)$ and $\tfrac1t\left(\mathcal T_t(\Phi)\cdot \nabla P_t^2-\mathcal T_t\right)$ in \cite[Lemma 5.5]{AHLMT}. In particular, for $I$ use that $S_t$ is a projection operator; i.e. $S_t^2 =S_t$,
\begin{multline*}
\|\tfrac1t\mathcal T_t(\Phi)\cdot (S_t-P_t^2)\left(\nabla\left(\chi f\right)\right)\|_{L^2(\R^{n+1}_+,dxdt/t)}\\
=\|\tfrac 1t\mathcal T_t(\Phi)\cdot(S_t+P_t)(S_t-P_t)\left(\nabla(\chi f)\right)\|_{L^2(\R^{n+1}_+,dxdt/t)}
\end{multline*}
that $\tfrac 1t\mathcal T_t(\Phi)\cdot (S_t+P_t)$ is a bounded operator on $L^2(\R^n)$ because the point wise bounds of the kernel of $S_t+P_t$ and duality show that
\begin{equation*}
\|\tfrac 1t\mathcal T_t(\Phi)(S_t+P_t)\|_{\mathcal B(L^2(\R^n))}\le N\|P_{Nt}\left(|\tfrac 1t\mathcal T_t(\Phi)|^2P_{Nt}\right)\|^{\frac 12}_{\mathcal B(L^2(\R^n))},
\end{equation*} 
while the first inequality in Lemma \ref{L: otrolemita} implies that
\begin{equation*}
\left(\fint_{B_{2Nt}(x)}|\tfrac 1t\mathcal T_t(\Phi)|^2\,dx\right)^{\frac 12}\le N,
\end{equation*}
which shows that the kernel of $P_{Nt}\left(|\tfrac 1t\mathcal T_t(\Phi)|^2P_{Nt}\right)$ is bounded by $Nt^{-n}\chi_{|x-y|\le 4Nt}$. Finally, the proof of the inequality
\begin{equation*}
\left(\int_{\R^{n+1}_+}|(S_t-P_t)(h)|^2\tfrac{dxdt}t\right)\le N\|h\|_{L^2(\R^n)},\ \text{for}\ h\in L^2(\R^n),
\end{equation*}
is explained in \cite{J} or \cite[pp. 168 and 172-173]{AT}. The bound for $II$ follows from \eqref{E: fundamental3} and Lemma \ref{L: elcachondeo} applied to the family of operators $\mathcal T_t(\Phi)\cdot \nabla P_t^2-\mathcal T_t$, which are uniformly bounded in $L^2(\R^n)$, verify Gaffney bounds and map $1$ and $\Phi$ to zero.

In order to bound $III$, the presence of the $BMO(\R^n)$ matrix $\mathbf D$, obliges to use some additional information about the gradient of $f=f^\e_{Q.\xi}$. In particular,  local higher integrability; i.e. there is $p=p(\lambda,n)>0$  independent of $\delta(Q)$ and $\xi$ such that
\begin{equation}\label{E:higherintegrability}
\left(\fint_{5Q}|\nabla f^\e_{Q,\xi}|^p\, dx\right)^{\frac 1p}\le N.
\end{equation}
Once the later is known, the skew-symmetry of $\mathbf D$ implies that as a distribution
\begin{equation*}
L\left(\chi f\right)=\chi Lf-\nabla\cdot\left(f\left(\mathbf A+\mathbf D-m_{4Q}(\mathbf D)\right)\nabla\chi\right)-\left(\mathbf A+\mathbf D-m_{4Q}(\mathbf D)\right)\nabla f\cdot\nabla\chi
\end{equation*}
and
\begin{equation}\label{E:necesario0}
\begin{split}
&\tfrac 1t{\mathcal T_t}(\chi f)=\left(1+t^2L\right)^{-1}tL\left(\chi f\right)\\
&=t\left(1+t^2L\right)^{-1}\left(\chi Lf\right)\\
&-\left(1+t^2L\right)^{-1}t\left[\nabla\cdot\left(f\left(\mathbf A+\mathbf D-m_{4Q}(\mathbf D)\right)\nabla\chi\right)+\left(\mathbf A+\mathbf D-m_{4Q}(\mathbf D)\right)\nabla f\cdot\nabla\chi\right]
\end{split}
\end{equation}
Then, \eqref{E: operadores basicos} and \eqref{E: propiedad3} give
\begin{equation}\label{E:necesaria}
\|t\left(1+t^2L\right)^{-1}\left(\chi Lf\right)\|_{L^2(Q)}\le Nt\delta(Q)^{-1}\sqrt{|Q|},
\end{equation}
while the Gaffney bounds in Lemma \ref{L: Gaffney2}, \eqref{eq1.1}, \eqref{E: condicion2}, \eqref{E:higherintegrability} \eqref{E: propiedad1}, \eqref{E: propiedad2}, H\"older's inequality and the Poincar\'e-Sobolev inequality over $4Q$, imply that for $0<t\le\delta(Q)$
\begin{equation}\label{E:necesario2}
\begin{split}
& \| \left(1+t^2L\right)^{-1}t\nabla\cdot\left( f\left(\mathbf A+\mathbf D-m_{4Q}(\mathbf D)\right)\nabla\chi\right)\|_{L^2(Q)}\\
&+\| \left(1+t^2L\right)^{-1}t\left[\left(\mathbf A+\mathbf D-m_{4Q}(\mathbf D)\right)\nabla f\cdot\nabla\chi\right] \|_{L^2(Q)}\\
&\le e^{-\delta(Q)/Nt}\delta(Q)^{-1} \left[\|\left(\mathbf A+\mathbf D-m_{4Q}(\mathbf D)\right)\left(f-m_{4Q}(f)\right)\|_{L^2(4Q)}+|m_{4Q}(f)|\sqrt{|Q|}\right]\\
&+ e^{-\delta(Q)/Nt} t\delta(Q)^{-1} \|\left(\mathbf A+\mathbf D-m_{4Q}(\mathbf D)\right)\nabla f\|_{L^2(4Q)}\le Ne^{-\delta(Q)/2Nt}\sqrt{|Q|}.
\end{split}
\end{equation}
Finally, \eqref{E:necesario0}, \eqref{E:necesaria}  and \eqref{E:necesario2} show that $III$ is also bounded by $N\sqrt{|Q|}$, which proves Lemma \ref{E: fundamental2}.

It only remains to show that \eqref{E:higherintegrability} holds; but this follows from \eqref{E: propiedad1} and \eqref{E: propiedad2} and standard higher integrability methods \cite{Gehring,Stredulinsky,SereginSilvestreSverakZlatos12} because $f=f^\e_{Q,\xi}$ is a weak $H^1_{loc}(\R^n)$ solution to $f+t^2Lf=\Phi_Q\cdot\xi$ over $\R^n$, with $t=\e \delta(Q)$. We include the details for the reader's convenience:

When $B_{2r}$ is  any ball, multiply the equation 
\begin{equation*}
-t^2\nabla\cdot\left(\left(\mathbf A+\mathbf D\right)\nabla f\right)=\Phi_Q\cdot\xi-f
\end{equation*}
by $\left(\overline f-m_{B_{2r}}(\overline f)\right)\eta^2$, with $\eta=1$ over $B_r$, $\eta$ in $C_0^{\infty}(B_{2r})$. It yields
\begin{multline*}
\int t^2\mathbf A\nabla f\cdot\nabla\overline f\eta^2+2t^2\left(\mathbf A+\mathbf D\right)\nabla f\cdot\nabla\eta\left(\overline f-m_{B_{2r}}(\overline f)\right)\eta\,dx\\
=\int\left(\Phi_Q\cdot\xi-f\right)\left(\overline f-m_{B_{2r}}(\overline f)\right)\eta^2\, dx.
\end{multline*}
Taking real parts,  dividing by $t^2$ and using the cancellations provided by the skew-symmetry of the matrix $m_{B_{2r}}(\mathbf D)$, one gets
\begin{multline*}
\int |\nabla f|^2\eta^2\, dx\le r^{-2}\int_{B_{2r}}\left(1+|\mathbf D-m_{B_{2r}}(\mathbf D)|^2\right)|f-m_{B_{2r}}(f)|^2\, dx\\
+\left(\int_{B_{2r}}|\left(\Phi_Q\cdot\xi-f\right)t^{-2}|^{\frac{2n}{n+2}}\, dx\right)^{\frac{n+2}{n}}.
\end{multline*}
Next, by H\"older's inequality and a Sobolev-Poincar\'e inequality
\begin{multline*}
\int_{B_{2r}}|\mathbf D-m_{B_{2r}}(\mathbf D)|^2|f-m_{B_{2r}}(f)|^2\, dx\\
\le \left(\int_{B_{2r}}|\mathbf D-m_{B_{2r}}(\mathbf D)|^{2n}\ dx\right)^{\frac 1n}\left(\int_{B_{2r}}|f-m_{B_{2r}}(f)|^{\frac {2n}{n-1}}\, dx\right)^{\frac{n-1}n}\\
\le Nr\left(\int_{B_{2r}}|\nabla f|^{\frac{2n}{n+1}}\, dx\right)^{\frac{n+1}n}.
\end{multline*}
Hence, recalling that $t=\e\delta(Q)$, one gets
\begin{multline}\label{E: reverseholder}
\left(\fint_{B_r}|\nabla f|^2\, dx\right)^{\frac 12}\le N\left(\fint_{B_{2r}}|\nabla f|^{\frac{2n}{n+1}}\, dx\right)^{\frac{n+1}{2n}}\\+N\left(\fint_{B_{2r}}|\left(\Phi_Q\cdot\xi-f\right)\delta(Q)^{-1}|^{\frac{2n}{n+2}}\, dx\right)^{\frac{n+2}{2n}},
\end{multline}
\noindent when $B_{2r}$ is any ball contained in  $10Q$. From \cite{Stredulinsky}  and \eqref{E: reverseholder}, there is some $p=p(\lambda, n)>2$ such that
\begin{multline}\label{E: reverseholder2}
\left(\fint_{5Q}|\nabla f|^p\, dx\right)^{\frac 1p}\le N\left(\fint_{10Q}|\nabla f|^2\, dx\right)^{\frac 12}\\+N\delta(Q)^{-1}\left(\fint_{10Q}|\Phi_Q\cdot\xi-f|^p\, dx\right)^{\frac 1p},
\end{multline}
Finally, we may assume that $2<p<\frac {2n}{n-2}$ and the interpolation of \eqref{E: propiedad1} and \eqref{E: propiedad2} shows that the second term in the right hand of \eqref{E: reverseholder2} is bounded by $N$, while \eqref{E: propiedad2} implies that the same holds with the first term.
\end{proof}
\end{section}

\end{document}